\documentclass[12pt]{article}

\usepackage{amsmath,amsfonts,amssymb}
\usepackage{bbm}
\newtheorem{defn}{Definition}[section]
\newtheorem{theo}[defn]{Theorem}
\newtheorem{lem}[defn]{Lemma}
\newtheorem{prop}[defn]{Proposition}

\newtheorem{rem}[defn]{Remark}
\newtheorem{exam}[defn]{Example}
\newenvironment{proof}{{\bf Proof }}{{\vskip 0.1cm \hfill$\Box$}}

\def\N {{\mathbb N}}

\def\R {{\mathbb R}}
\def\E{{\mathbb E}}
\def\P{{\mathbb P}}
\def\M{{\mathbb M}}
\newcommand{\F}{\mathcal{F}}

\begin{document}

\noindent
{{\Large\bf Existence and regularity of infinitesimally invariant measures, transition functions and time homogeneous It\^o-SDEs}
{\footnote{This research was supported by Basic Science Research Program through the National Research Foundation of Korea(NRF) funded by the Ministry of Education(NRF-2017R1D1A1B03035632).}}\\ \\
\bigskip
\noindent
{\bf Haesung Lee},
{\bf Gerald Trutnau}
}\\

\noindent
{\bf Abstract.} We show existence of an infinitesimally invariant measure $m$ for a large class of divergence and non-divergence form elliptic second order partial differential operators with locally Sobolev regular diffusion coefficient and drift of some local integrability order. Subsequently, we derive regularity properties of the corresponding semigroup which is defined in $L^s(\mathbb{R}^d,m)$, $s\in [1,\infty]$, including the classical strong Feller property and classical irreducibility. This leads to a transition function of a Hunt process that is explicitly identified as a solution to an SDE. Further properties of this Hunt process, like non-explosion, moment inequalities, recurrence and transience, as well as ergodicity, including invariance and uniqueness of $m$, and uniqueness in law, can then be studied using the derived analytical tools and tools from generalized Dirichlet form theory. 

\section{Introduction}
Throughout, we let the dimension $d\ge 2$. 
We investigate a quite general class of divergence form operators with respect to a possibly non-symmetric diffusion matrix $(a_{ij}+c_{ij})_{1\le i,j\le d}$ and perturbation $\mathbf{H}=(h_1,\ldots,h_d)$, which can be written as
\begin{eqnarray}\label{defofLdivform}
Lf & = & \frac12 \sum_{i,j=1}^{d} \partial_i((a_{ij}+c_{ij})\partial_j)f+\sum_{i=1}^{d}h_i\partial_i f, \quad f \in C_0^{\infty}(\R^d).
\end{eqnarray}
Here, we consider the {\bf assumption}
\begin{itemize}
\item[(a)] $A= (a_{ij})_{1\le i,j\le d}$ is a $d \times d$ matrix of functions, such that $a_{ji}=a_{ij}\in H_{loc}^{1,2}(\R^d) \cap C(\R^d)$ for all $1 \leq i, j \leq d$ and such that for every open ball $B \subset \R^d$, there exist positive real numbers $\lambda_{B}$, $\Lambda_{B}$ with
\begin{eqnarray}\label{uniformell}
\lambda_{B} \| \xi \|^2 \leq \langle A(x) \xi, \xi \rangle \leq \Lambda_{B} \| \xi \|^2 \text{ for all } \xi \in \R^d, x \in B.
\end{eqnarray}
$\mathbf{H}=(h_1, \ldots, h_d) \in L_{loc}^p(\R^d, \R^d)$, i.e. $h_i\in  L_{loc}^p(\R^d)$, $1\le i\le d$, for some $p>d$, and $C = (c_{ij})_{1\le i,j\le d}$ is a  $d \times d$ matrix of functions, with $-c_{ji}=c_{ij} \in H_{loc}^{1,2}(\R^d) \cap C(\R^d)$  for all $1 \leq i,j \leq d$,
\end{itemize}
and the  {\bf assumption}
\begin{itemize}
\item[(b)] $\frac{1}{2}\nabla \big (A+C^{T}\big )+ \mathbf{H} \in L_{loc}^q(\R^d, \R^d)$, where throughout $q:= \frac{pd}{p+d}$,
\end{itemize}
on the coefficients of $L$.\\
Our first observation is that just under assumption (a), there exists a density $\rho$, which determines an infinitesimally invariant measure $m=\rho\,dx$ for $(L,C_0^{\infty}(\R^d))$, and which has a nice regularity (see Theorem \ref{eim}). This extends \cite[Theorem 1(i)]{BRS} (cf. Remark \ref{ex1}, where it is also shown that such operators cover a fairly general class of non-divergence form operators) and leads by a construction method of \cite{St99} to a $C_0$-semigroup of sub-Markovian contractions $(T_t)_{t\ge 0}$ on $L^1(\R^d,m)$, whose generator is an extension of $(L,C_0^{\infty}(\R^d))$, i.e. we have found a suitable functional analytic frame for $(L,C_0^{\infty}(\R^d))$. This functional analytic frame is also described by a generalized Dirichlet form. Subsequently in Section \ref{two}, we investigate the regularity properties of the semigroup $(T_t)_{t\ge 0}$ and its corresponding resolvent $(G_{\alpha})_{\alpha>0}$, which can in fact be considered in every $L^s(\R^d,m)$, $s\in [1,\infty]$. The regularity properties comprise strong Feller properties, i.e. the existence of continuous versions $P_t f$, $f\in L^\infty(\R^d,m)+L^1(\R^d,m)$ and $R_{\alpha}g$, $g\in L^\infty(\R^d,m)+L^q(\R^d,m)$, of $T_t f$ and $G_{\alpha} g$, as well as the irreducibility of $(P_t)_{t>0}$ (Lemma \ref{irreduci0}(i)).\\
In Section \ref{4}, we investigate the stochastic counterpart of $(P_t)_{t>0}$. Adding just assumption (b) to assumption (a) suffices to obtain that $(P_t)_{t>0}$  is the transition function of a Hunt process $\M$ and to carry over most of the probabilistic results from \cite{LT18} to the more general situation considered here (see Remark \ref{mostresults} and Theorem \ref{weakexistence} which states  that $\M$ solves weakly the stochastic differential equation with coefficients given by $L$). In Theorem \ref{supestimate}, we present a new non-explosion condition, which leads to a moment  inequality. It also allows for $L^q(\R^d,m)$-singularities outside an arbitrarily large compact set and linear growth of the drift at the same time. An application of  Theorem \ref{supestimate} is illustrated in the Example \ref{infsin}. In Section \ref{4.2}, we discuss the relation of $L^1(\R^d,m)$-uniqueness from  \cite{St99}, the strong Feller property derived here and uniqueness in law. More precisely, we obtain a result on uniqueness in law among all right processes that have $m$ as sub-invariant measure (see Propositions \ref{wilhelm} and \ref{uniquenesslawgen}). \\ 
Finally, we would like to discuss a special aspect of our work, which we think  is remarkable and to relate our work to some other references. The Hunt process $\M$ which is constructed in this article satisfies the following Krylov type estimate: let $g\in L^r(\R^d,m)$ for some $r\in [q,\infty]$. Then for any Euclidean ball $B$ there exists a constant $c_{B,r,t}$, depending in particular on $B$, $t$, and $r$, but not on $g\in L^r(\R^d, m)$, $g\ge0$, such that for all $t \ge  0$
\begin{equation} \label{resest}
\sup_{x\in \overline{B}}\E_x\left [ \int_0^t g(X_s) \, ds \right ] < c_{B,r,t}\, \|g\|_{L^r(\R^d, m)}.
\end{equation}
Using Theorem \ref{4.1} below, \eqref{resest} can be shown exactly as in \cite[Lemma 3.14(ii)]{LT18}. Such type of estimate is an important tool for the analysis of diffusions (see for instance \cite{Kry} and in particular \cite[p.54, 4. Theorem]{Kry} for the original estimate involving conditional expectation, or also \cite{GyMa} and \cite{Zh11}). A priori  \eqref{resest} only holds for the Hunt process $\M$ constructed here. However, if pathwise uniqueness holds (for instance if the coefficients here are locally Lipschitz or under the conditions in \cite{Zh11}), or more generally uniqueness in law holds for the SDE solved by $\M$ with certain given coefficients, then \eqref{resest} holds generally for any diffusion with the given coefficients. If further $g\in L^r(\R^d)$ has compact support, then $\|g\|_{L^r(\R^d, m)}$ in \eqref{resest} can be replaced by $\|g\|_{L^r(\R^d)}$, when $c_{B,r,t}$ is replaced by a constant $c_{B,r,t,\rho}$ that also depends on the values of $\rho$ on the support of $g$. If $A$, $C$, $\rho$, $\widetilde{\mathbf{B}}$ are explicitly given, as described in Remark \ref{expg}(i), i.e. the case where the generalized Dirichlet form is explicitly given as in \cite{St99}, then \eqref{resest} holds with explicit $\rho$ and \eqref{resest} can be seen as a Krylov type estimate for a large class of time-homogeneous generalized Dirichlet forms. As a particular example consider the non-symmetric divergence form case, i.e. the case where $\mathbf{H}\equiv 0$ in \eqref{defofLdivform}. Then the explicitly given $\rho\equiv 1$ defines an infinitesimally invariant measure. Hence $m$ in \eqref{resest} can be replaced by Lebesgue measure in this case. The latter together with some further results of this article complement analytically as well as probabilistically aspects of the works \cite{Str88},  \cite{Roz96}, and \cite{TaTr} where also divergence form operators are treated, but where more emphasis is put on the mere measurability of the diffusion matrix and not on the generality of the drift.

\section{Terminologies and notations}
For a  matrix $A$, let $A^{T}$ denote the transposed matrix of $A$. If $A=(a_{ij})_{1\le i,j\le d}$ consists of weakly differentiable functions $a_{ij}$, we define
$$
\nabla A=((\nabla A)_1,\ldots , (\nabla A)_d),\qquad (\nabla A)_i:=\sum_{j=1}^{d} \partial_{j} a_{ij},\quad  1 \leq i \leq d.
$$
If $f$ is two times weakly differentiable, let $\nabla^2 f$ denote the Hessian matrix of second order weak partial derivatives of $f$. In particular
$$
\text{trace}(A\nabla^2 f)=\sum_{i,j=1}^{d} a_{ij}\partial_i \partial_j f.
$$
If $\rho$ is weakly differentiable and a.e. positive then
$$
\beta^{\rho, A}=(\beta_1^{\rho, A}, \ldots, \beta_d^{\rho, A}):= \frac{1}{2}\Big(\nabla A+\frac{A\nabla \rho}{\rho} \Big),
$$
is called the logarithmic derivative of $\rho$ associated with $A$. Hence
$$
\beta^{\rho,A^T}_i=\frac12\sum_{j=1}^{d}\Big (\partial_j a_{ji}+a_{ji}  \frac{\partial_j \rho}{\rho}\Big), \quad 1\le i\le d.
$$
For a bounded open subset $U$ of $\R^d$ and a possibly non-symmetric matrix of measurable functions $A=(a_{ij})_{1 \leq i,j \leq d}$ on $U$, we say that $A$ is {\it uniformly strictly elliptic and bounded} on $U$, if there exist $\lambda>0$ and $M>0$ such that for any $\xi=(\xi_1, \dots, \xi_d) \in \R^d$, $x\in U$,
$$
\sum_{i,j=1}^{d} a_{ij}(x) \xi_i \xi_j \geq \lambda \|\xi\|^2, \qquad \max_{1 \leq i, j \leq d} | a_{ij}(x)   | \leq M.
$$
In that case, $\lambda$ is called the ellipticity constant and $M$ is called the upper bound constant of $A$.
For other definitions or notations that might be unclear, we refer to \cite{LT18}.
\section{Analytic results}
\subsection{Elliptic  $H^{1,p}$-regularity and $H^{1,p}$-estimates}
The $VMO(\R^d)$ space is defined as the space of all locally integrable functions $f$ on $\R^d$ for which there exists a positive continuous function $\gamma$ on $[0, \infty)$ with $\gamma(0)=0$, such that 
\begin{equation} \label{defVMO}
\sup_{z \in \R^d, r <R} r^{-2d} \int_{B_{r}(z) \times B_{r}(z)} |f(x)-f(y)| dx dy \leq \gamma(R),\;\; \forall R>0.
\end{equation}
If $f$ is uniformly continuous on $\R^d$, we can define 
$$
\gamma(r):=\displaystyle \Big(\int_{B_1} 1\,dx\Big)^{-2} \cdot \displaystyle \sup_{|x-y|<2r, x,y \in \R^d} |f(x)-f(y)|, \quad\gamma(0):=0.
$$ 
Then  $\gamma$ is continuous on $[0, \infty)$ and \eqref{defVMO} holds, hence $f \in VMO(\R^d)$. For a bounded open subset $U$ of $\R^d$ and a function $g$ on $U$, we call $g \in VMO(U)$ if $g$ extends to a function on $\R^d$, again called $g$, such that $g \in VMO(\R^d)$.\\ \\
For measurable functions $a_{ij}$, $b_i$, $\beta_i$, $c$ on $\R^d$, $1 \leq i, j \leq d$, let $A:=(a_{ij})_{1 \leq i,j \leq d}$, $b:=(b_1, \dots, b_d)$, $\beta:=(\beta_1, \dots, \beta_d)$.  
Consider the divergence form operator $\mathcal{L}$, defined in distribution sense
$$
-\mathcal{L}u := -\left(\sum_{i,j=1}^{d} \partial_i(a_{ij} \partial_j u )  +\sum_{i=1}^{d} \partial_i( b_i u) \right) + \sum_{i=1}^{d}\beta_i \partial_i u + cu, \quad u\in C^{\infty}_0(\R^d).
$$
The following theorem is a simple generalization of (1.2.3) in  \cite[Theorem 1.2.1]{BKRS}, where only symmetric matrices of functions are considered.
\begin{theo}[Krylov 2007]\label{1.1}
Consider a possibly non-symmetric matrix of functions $A=(a_{ij})_{1\le i,j\le d}$ and suppose that $a_{ij} \in VMO(\R^d)$, $1\le i,j \le d$, and that there exist $\varepsilon, K>0$ such that 
\begin{align*}
&\qquad \qquad \sum_{i,j=1}^{d} a_{ij}(x) \xi_i \xi_j \geq \varepsilon \|\xi\|^2_{\R^d} \text{ for all } \xi \in \R^d, \; \text{ a.e. } x \in \R^d,\\
& \sum_{i,j=1}^{d} \|a_{ij} \|_{L^\infty(\R^d)}+\sum_{i=1}^{d} \|b_i\|_{L^\infty(\R^d)}+\sum_{i=1}^{d} \|\beta_i\|_{L^\infty(\R^d)}+\|c\|_{L^\infty(\R^d)}  \leq K.
\end{align*}
Then, for every $p \in (1, \infty)$, there are numbers $\lambda_0$ and $M$ depending only $p,d, K, \varepsilon$ and a common $\gamma$ that ensures the $VMO(\R^d)$ condition \eqref{defVMO} simultaneously for all $a_{ij}$, $1 \leq i,j \leq d$, such that for all $\lambda \geq \lambda_0$, $v \in H^{1,p}(\R^d)$, we have
$$
\|v\|_{H^{1,p}(\R^d)} \leq  M  \|\mathcal{L}v - \lambda v  \|_{H^{-1,p}(\R^d)}.
$$
\end{theo}
\begin{proof}
Take constants $\lambda_0$, $N$ as in \cite[Theorem 2.8]{Kr07}, which depend only on $p,d, K, \varepsilon$. Let $\lambda\ge\lambda_0$ be given.
By \cite[Proposition 9.20]{BRE}, there exist $f \in L^p(\R^d)$ and $g=(g_1, \dots, g_d)\in L^p(\R^d, \R^d)$ such that
\begin{equation} \label{divform}
\mathcal{L}v - \lambda v = f + {\rm div} g \quad \text{ in } H^{-1,p}(\R^d),
\end{equation}
where
$$
\|\mathcal{L}v - \lambda v  \|_{H^{-1,p}(\R^d)} = \max(\|f\|_{L^p(\R^d)}, \|g_1\|_{L^p(\R^d)}, \dots, \|g_d\|_{L^p(\R^d)}).
$$
Thus 
$$
\displaystyle \|f\|_{L^p(\R^d)} + \sum_{i=1}^{d} \|g_i\|_{L^p(\R^d)} \leq (d+1) \|\mathcal{L}v - \lambda v  \|_{H^{-1,p}(\R^d)}.
$$
By \cite[Theorem 2.8]{Kr07} $v$ is the unique solution to \eqref{divform} and
\begin{eqnarray*}
\|v\|_{H^{1,p}(\R^d)} &\leq& N \left(\|f\|_{L^p(\R^d) } + \sum_{i=1}^{d} \|g_i\|_{L^p(\R^d)}   \right) \\
&\leq& \underbrace{N(d+1)}_{=:M} \|\mathcal{L}v - \lambda v  \|_{H^{-1,p}(\R^d)}.
\end{eqnarray*}
\end{proof}\\
We shall make a general remark concerning the monograph \cite{BKRS}.
\begin{rem}\label{nonsym}
In what follows, we shall use in particular the statements 1.7.4, 1.7.6, 1.8.3,  2.1.4, 2.1.6, 2.1.8 of \cite{BKRS} which are formulated for a symmetric matrix of functions $A=(a_{ij})_{1\le i,j \le d}$ on a bounded smooth domain $\Omega$, such that each function $a_{ij}$ is $VMO(\Omega)$ and $A$ is uniformly strictly elliptic and bounded on $\Omega$. However, a closer look at the corresponding proofs shows that the symmetry is not a necessary  assumption. More precisely,  (1.7.10) in the proof of \cite[Theorem 1.7.4]{BKRS} follows from (1.2.3) of \cite[Theorem 1.2.1]{BKRS}. But by a result of Krylov the symmetry of $(a_{ij})_{1 \leq i,j \leq d}$ is not essential in Theorem \ref{1.1}.  Consequently, \cite[Corollary 1.7.6]{BKRS}, whose proof is based on  \cite[Theorem 1.7.4]{BKRS}, also holds for a non-symmetric matrix of functions  $(a_{ij})_{1 \leq i, j \leq d}$ which is uniformly strictly elliptic and bounded on $\Omega$. The proof of \cite[Proposition 2.1.4]{BKRS} is based on the Lax-Milgram Theorem which only uses a coercivity assumption that is well-known to extend to a non-symmetric matrix of functions. \cite[Theorem 2.1.8]{BKRS} is taken from \cite{T77}, where not only non-symmetric matrices of functions are permitted but also even more general conditions on the functions $a_{ij}, 1\leq i, j \leq d$. \cite[Corollary 2.1.6]{BKRS} is a consequence of \cite[Corollary 1.7.6 , Proposition 2.1.4 and Theorem 2.1.8]{BKRS}.  Finally, the proof of
\cite[Theorem 1.8.3]{BKRS} follows from \cite[Corollary 1.7.6 and Proposition 2.1.4]{BKRS}.  
Therefore all the above mentioned statements  from \cite{BKRS} extend to a non-symmetric matrix of functions $A=(a_{ij})_{1\le i,j \le d}$, such that each function $a_{ij}$ is $VMO(\Omega)$ and $A$ is uniformly strictly elliptic and bounded on $\Omega$. 
However, we will assume more than $VMO(\Omega)$, more precisely $H_{loc}^{1,2}(\R^d) \cap C(\R^d)$, in  what follows since we need an integration by parts formula.
\end{rem}
The following Lemma \ref{compactness} will be used in the proof of Lemma \ref{subsol} for a compactness argument.
\begin{lem}\label{compactness}
Let $A = (a_{ij})_{1 \leq i,j \leq d}$, $A_n=(a^n_{ij})_{1 \leq i,j \leq d}$  be uniformly strictly elliptic and bounded on an open ball $B$, satisfying $a^{n}_{ij} \to a_{ij}$ in $L^2(B)$ as $n \to \infty$, $1 \leq i,j \leq d$. Moreover, let $A_n$, $n\in \N$, and $A$ have the same ellipticity constant $\lambda_n\equiv \lambda$ and upper bound constant $M_n\equiv M$.
Let for some $p>d$, $b \in L^p(B, \R^d)$, $b_n \in L^p(B, \R^d)$ such that $b_n \to b$ in $L^p(B, \R^d)$ as $n \to \infty$. Given $F \in L^2(B,\R^d)$, suppose that $u_{n,F} \in H_0^{1,2}(B)$ satisfies
$$
\int_{B} \langle A_n \nabla u_{n,F}+b_n u_{n,F}, \nabla \varphi \rangle \,dx = \int_B \langle F, \nabla \varphi  \rangle \,dx, \quad \text{ for every } \varphi \in C_0^{\infty}(B).
$$
Then
$$
\|u_{n,F}\|_{L^2(B)} \leq C \|F\|_{L^2(B, \R^d)},
$$
where $C>0$ is a constant which is independent of $n$ and $F$.
\end{lem}
\begin{proof}
Assume that the assertion does not hold, i.e. given $k \in \N$ there exist $\widetilde{F}_k \in L^2(B, \R^d)$  and $n_k \in \N$ such that
$$
\|u_{n_k,\widetilde{F}_k}\|_{L^2(B)} > k \|\widetilde{F}_k\|_{L^2(B, \R^d)}.
$$
Define $F_k:=\displaystyle \frac{ \widetilde{F}_k}{ \|u_{n_k,\widetilde{F}_k}\|_{L^2(B)}  }$. By \cite[Proposition 2.1.4, Theorem 2.1.8]{BKRS} and Remark \ref{nonsym}, we get $u_{n_k, F_k}=\displaystyle \frac{ u_{n_k,\widetilde{F}_k}}{ \|u_{n_k,\widetilde{F}_k}\|_{L^2(B)}  }$. Thus we have
$$
\|u_{n_k,F_k}\|_{L^2(B)}=1\;\; \text{ and } \quad\|F_k\|_{L^2(B, \R^d)} < \frac{1}{k}.
$$
By \cite[Th\'eor\`eme 3.2]{St65},
\begin{eqnarray*}
\|u_{n_k, F_k}\|_{H_0^{1,2}(B)} &\leq& C_1(\|u_{n_k, F_k}\|_{L^2(B)} + \|F_k\|_{L^2(B, \R^d)})\ \leq \ 2C_1,
\end{eqnarray*}
where $C_1$ is independent of $k$.
By the weak compactness of balls in $H_0^{1,2}(B)$ and the Rellich-Kondrachov Theorem, 
there exist a subsequence $( u_{n_{k_j}, F_{k_j}} )_j \subset (u_{n_k, F_{k}})_k$ and $u \in H_0^{1,2}(B)$ such that
$$
u_{n_{k_j} , F_{k_j}} \rightarrow  u \ \text{ weakly in } H_0^{1,2}(B), \quad u_{n_{k_j}, F_{k_j}} \rightarrow u \ \text{ in } L^2(B).
$$
In particular, $\|u\|_{L^2(B)}=1$ and using the assumption, we can see that $u$ satisfies
$$
\int_{B} \langle A\nabla u+b u, \,\nabla \varphi \rangle dx = 0, \quad \text{ for every } \varphi \in C_0^{\infty}(B).
$$
By \cite[Theorem 2.1.8]{BKRS} and Remark \ref{nonsym}, we have $u=0$ a.e. on $B$, which is a contradiction. Therefore the assertion must hold.
\end{proof}\\
The following is well known in the case where $b\equiv 0$ (see for instance \cite[Lemma 4.6]{HanLin}). 
\begin{lem}\label{subsol}
Let $U$ be a bounded open set with Lipschitz boundary. Let $A = (a_{ij})_{1 \leq i,j \leq d}$ be uniformly strictly elliptic and bounded on $U$, with ellipticity constant $\lambda$ and upper bound constant $M$. Let for some $p>d$, $b \in L^p(U, \R^d)$ and assume that $u \in H^{1,2}(U)$ satisfies
$$
\int_{U} \langle A \nabla u+bu, \nabla \varphi \rangle dx \leq 0, \quad \text{ for every } \varphi \in C_0^{\infty}(U),\; \varphi \geq 0.
$$
Then we have
$$
\int_{U} \langle A \nabla u^{+}+bu^{+}, \nabla \varphi \rangle dx \leq 0, \quad \text{ for every } \varphi \in C_0^{\infty}(U),\; \varphi \geq 0.
$$
\end{lem}
\begin{proof}
Let $B$ be an open ball such that $\overline{U} \subset B$. By \cite[Theorem 4.7]{EG15}, $u \in H^{1,2}(U)$ can be extended to a function $u\in H_0^{1,2}(B)$. And by \cite[Theorem 4.4]{EG15}, $u^{+} \in H_0^{1,2}(B)$ with
$$
\nabla u^+=
\left\{\begin{matrix}
\nabla u \quad \text{ a.e. on } \left\{ u>0  \right\}, \\[5pt]
\; 0 \quad \quad \text{ \;a.e. on } \left\{ u \leq 0  \right\}.
\end{matrix}\right.
$$ 
Given $\varepsilon>0$ define \;
$$
f_{\varepsilon}(z):=
\left\{\begin{matrix}
\;(z^2+\varepsilon^2)^{1/2}- \varepsilon \quad \text{ if } z \geq 0, \\[5pt]
\;\;0 \qquad \qquad  \qquad  \quad \text{ if } z<0.
\end{matrix}\right.
$$
Then $f_{\varepsilon} \in C^1(\R)$, $f_{\varepsilon}' \in  H^{1, \infty}(\R)$, and
$$
f'_{\varepsilon}(z)= 
\left\{\begin{matrix}
\displaystyle \frac{z}{\sqrt{z^2+\varepsilon^2}} \quad \text{ if } z \geq 0, \\[8pt]
0 \quad \quad \qquad  \;\;\text{ if } z<0,
\end{matrix}\right.
\qquad and \qquad 
f''_{\varepsilon}(z)= 
\left\{\begin{matrix}
\displaystyle \frac{\varepsilon^2}{(z^2+\varepsilon^2)^{3/2}} \quad \text{ if } z >0, \\[8pt]
0 \quad \quad \qquad  \quad \;\;\text{ if } z<0.
\end{matrix}\right.
\qquad
$$
Note that $f_{\varepsilon}(z) \longrightarrow z^{+}$, $f'_{\varepsilon}(z) \longrightarrow 1_{(0,\infty)}(z)$ as $\varepsilon \rightarrow 0$ for every $z \in \R$. Extend the matrix of functions $A$ to whole $\R^d$ with same ellipticity constant $\lambda$ and upper bound constant $M$. (This is possible, for instance set $A=\lambda\cdot Id$ on $\R^d\setminus U$ and note that $\lambda\le M$.) Extend $b \in L^p(U, \R^d)$ to $L^p(\R^d, \R^d)$ by setting $b$ zero outside $U$. Define $F:=A \nabla u+bu \in L^2(\R^d, \R^d)$. For $n\in \N$ let $\eta_{\frac{1}{n}} \in C_0^{\infty}(B_{\frac{1}{n}})$ be defined as usually through the standard mollifier and let $a^n_{ij}:= a_{ij} * \eta_{\frac{1}{n}}$, $A_n:=(a^n_{ij})_{1 \leq i, j \leq d}$,  $b_n:= b* \eta_{\frac{1}{n}}$, $F_n:= F * \eta_{\frac{1}{n}}$ on $\R^d$. Then $a^n_{ij }\in C^{\infty}(\overline{B})$, $b_n, F_n \in C^{\infty}(\overline{B}, \R^d)$ satisfy
\begin{equation} \label{stcon}
a^n_{ij} \longrightarrow a_{ij}, \text{ in } \; L^2(B), \;   \quad b_n \longrightarrow b\; \text{ in } \; L^p(B, \R^d), \quad F_n \longrightarrow F \; \text{ in } \; L^2(B, \R^d).
\end{equation}
Moreover, each $A_n$, $n\in \N$, is uniformly strictly elliptic and bounded on $B$ with same elliptic constant $\lambda$ and upper bound constant $M$ as $A$. Let $V$ be a fixed open set with $\overline{V} \subset U$. Choose $\delta>0$ with $\overline{B}_{\delta}(z) \subset U$ for all $z \in V$ and take $N \in \N$ with $\frac{1}{N} < \delta$. Then by the assumption, for any $n \geq N$ and $\varphi \in C_0^{\infty}(V)$ with $\varphi \geq 0$
\begin{eqnarray}\label{ineq1}
\int_{U} \langle F_n, \nabla \varphi \rangle dx = \int_{U}  \langle A \nabla u+bu, \nabla (\varphi* \eta_{\frac{1}{n}})  \rangle \,dx \leq 0.
\end{eqnarray}
By \cite[Proposition 2.1.4, Theorem 2.1.8]{BKRS} and Remark \ref{nonsym}, there exists $u_n \in H_0^{1.2}(B)$ such that
\begin{eqnarray}\label{ineq2}
\int_{B} \langle A_n \nabla u_n+b_n u_n, \nabla \widetilde{\varphi} \rangle dx = \int_{B} \langle F_n, \nabla  \widetilde{\varphi} \rangle dx, \quad \text{ for all } \; \widetilde{\varphi} \in C_0^{\infty}(B).
\end{eqnarray}
By \cite[Th\'eor\`eme 3.2]{St65} and Lemma \ref{compactness},
\begin{eqnarray*}
 \|u_n\|_{H_0^{1,2}(B)} &\leq& C_1 \|F_n\|_{L^2(B, \R^d)} \ \leq \ C_1\|F\|_{L^2(B,\R^d)}.
\end{eqnarray*}
where $C_1$ is independent of $n$.
By weak compactness of balls in $H_0^{1,2}(B)$, \cite[Theorem 2.1.8]{BKRS} and Remark \ref{nonsym}, there exists a subsequence $(u_{n_k} )_k \subset (u_{n})_n$, such that
\begin{eqnarray} \label{weak+}
u_{n_k} \rightarrow u \ \text{ and }  \ u^+_{n_k} \rightarrow u^+ \ \text{ weakly in } \ H^{1,2}_0(B). 
\end{eqnarray}
Indeed,  \eqref{weak+} first holds with $u$ replaced by some $\widetilde{u}\in H^{1,2}_0(B)$. Then letting $n\to \infty$ in \eqref{ineq2}  and using the maximum principle \cite[Theorem 4]{T77}, we get $\widetilde{u}=u$. For simplicity, write $(u_n)$ for $(u_{n_k})$. By \cite[Theorem 8.13]{Gilbarg}, we have $u_n \in C^{\infty}(\overline{B})$. Now define
$$
\mathcal{L}_n u_n:=\sum_{i,j=1}^{d}a^n_{ij} \partial_i \partial_j u_n + \langle b_n +\nabla A^T_n, \nabla u_n \rangle + ({\rm div}\, b_n) \cdot u_n
$$
Then for any $n \geq N$ and $\varphi \in C_0^{\infty}(V)$ with $\varphi \geq 0$, we obtain using \eqref{ineq1}, \eqref{ineq2}
\begin{eqnarray*}
-\int_{U}\mathcal{L}_n u_n  \, \varphi \,dx \leq 0.
\end{eqnarray*}
Hence $\mathcal{L}_n u_n(x) \geq 0$ for all $x \in V$, $n \geq N$. 
Define $f_{\varepsilon}^k:=f_{\varepsilon}*\phi_{\frac{1}{k}}$, $k\in \N$, where $\phi_{\frac{1}{k}}\in C_0^{\infty}\big ((-\frac{1}{k},\frac{1}{k})\big )$ is the standard mollifier. Then $(f^k_{\varepsilon})' \geq 0$, $(f^k_{\varepsilon})'' \geq 0$ since $f'_{\varepsilon} \geq 0$, $f''_{\varepsilon} \geq 0$. Moreover, $(f^k_{\varepsilon})'(u_n)\to f'_{\varepsilon}(u_n)$ uniformly on $U$ 
as $k\to \infty$. Then,  for any $n \geq N$ and $\varphi \in C_0^{\infty}(V)$ with $\varphi \geq 0$, we obtain
\begin{eqnarray*}
 &&\int_{U} \langle A_n \nabla f_{\varepsilon}(u_n)+b_n f_{\varepsilon}(u_n), \nabla \varphi \rangle dx \ = \ \lim_{k\to \infty}\int_{U}\langle A_n \nabla f^k_{\varepsilon}(u_n)+b_n f^k_{\varepsilon}(u_n), \nabla \varphi \rangle dx \\
 &=& \lim_{k\to\infty}\Big (-\int_{U} \big ( (f^k_{\varepsilon})'(u_n) \mathcal{L}_n u_n +( f^k_{\varepsilon})''(u_n) \langle A_n \nabla u_n, \nabla u_n \rangle\big ) \cdot \varphi \,dx\Big)\\
 && -\lim_{k\to\infty}\int_{U} {\rm div}\, b_n ( f^k_{\varepsilon} (u_n) - u_n (f^k_{\varepsilon})'(u_n) )  \cdot \varphi \,dx \\
&\leq& - \int_{U} \;{\rm div}\, b_n \big( f_{\varepsilon} (u_n) - u_n f_{\varepsilon}'(u_n) \big)  \varphi dx.
\end{eqnarray*}
Since the latter term converges to zero as $\varepsilon \to 0$, for any $n \geq N$, we obtain
$$
 \int_{U}  \langle A_n \nabla u_n^{+}+b_n u_n^{+}, \nabla \varphi  \rangle dx \leq 0, \quad \forall \varphi \in C_0^{\infty}(V),\; \varphi \geq 0.
$$
Consequently, using  \eqref{stcon}, \eqref{weak+}, we get
$$
 \int_{U} \langle A \nabla u^{+}+b u^{+}, \nabla \varphi \rangle dx \leq 0, \quad \forall \varphi \in C_0^{\infty}(V),\; \varphi \geq 0.
$$
Since $V$ is an arbitrary open set with $\overline{V} \subset U$, the assertion follows.
\end{proof}

\subsection{Existence of an infinitesimally invariant measure and construction of a generalized Dirichlet form}\label{one}
We first start with a remark, that clarifies the relation of the divergence type operator \eqref{defofLdivform} and a fairly general class of non-divergence form operators. Moreover, we give some examples of operators satisfying assumption \text{(a)}.
\begin{rem}\label{ex1}
Note that under assumption \emph{(a)},  $L$ as in \eqref{defofLdivform} writes for $f\in C_0^{\infty}(\R^d)$ as
\begin{eqnarray}\label{BRSgen}
Lf &= & \frac12\emph{div}\big ( (A+C)\nabla f\big )+\langle\mathbf{H}, \nabla f\rangle \nonumber\\
&= & \frac12\emph{trace}\big ( A\nabla^2 f\big )+\big \langle\frac{1}{2}\nabla  (A+C^{T})+ \mathbf{H}, \nabla f\big \rangle. 
\end{eqnarray}
Thus $L$ as in \eqref{defofLdivform}  can also be interpreted as non-divergence form operator and therefore, assumption \emph{(a)} allows to consider two general classes of operators:
\begin{itemize}
\item[(i)] \emph{Divergence type operators as in \eqref{defofLdivform} with symmetric or nonsymmetric matrix and with or without $L_{loc}^p$-drift, according to assumption \text{(a)}}: for instance 
$$
Lf = \frac12 \sum_{i,j=1}^{d} \partial_i((a_{ij}+c_{ij})\partial_j)f+\sum_{i=1}^{d}h_i\partial_i f, \quad f \in C_0^{\infty}(\R^d),
$$
or
$$
Lf = \frac12 \sum_{i,j=1}^{d} \partial_i((a_{ij}+c_{ij})\partial_j)f, \quad f \in C_0^{\infty}(\R^d),
$$
where $a_{ij}$, $c_{ij}$ and $h_i$ satisfy assumption \emph{(a)} and $(c_{ij})_{1\le i,j\le d}\equiv 0$ or not.
\item[(ii)] \emph{Non-divergence type operators with symmetric diffusion matrix and $L_{loc}^p$-drift}: for this, suppose that $a_{ij} \in  H_{loc}^{1,p}(\R^d)\cap C(\R^d)$, $1\le i,j\le d$, for some $p>d$, and that $C\equiv 0$. Set
$$
\mathbf{H}:=\mathbf{\widetilde{H}}-\frac12 \nabla A
$$
for arbitrarily chosen $\mathbf{\widetilde{H}}=(\widetilde{h}_1,\ldots,\widetilde{h}_d)\in  L_{loc}^p(\R^d,\R^d)$. Then assumption \emph{(a)} (and even assumption \emph{(b)}) holds (since $p>q$) and \eqref{defofLdivform} can be rewritten as
\begin{eqnarray}\label{defofLnondivform}
Lf & = & \frac12 \sum_{i,j=1}^{d} a_{ij}\partial_{ij} f+\sum_{i=1}^{d}\widetilde{h}_i\partial_i f, \quad f \in C_0^{\infty}(\R^d).
\end{eqnarray}
This special case covers the assumptions of \cite[Theorem 1]{BRS} (see also \cite[Theorem 2.4.1]{BKRS}). In general, we can consider any non-divergence type operator  as in \eqref{BRSgen}, where $A,C$, and $\mathbf{H}$ satisfy the assumption \emph{(a)}. The latter, together with the class of divergence form opertors considered in (i), is the extend to which we can generalize the assumptions of  \cite[Theorem 1]{BRS}.
\end{itemize}
\end{rem}
{\bf From now on}, we set
$$
\mathbf{G}=(g_1, \ldots, g_d)=\frac{1}{2}\nabla \big (A+C^{T}\big )+ \mathbf{H},
$$ 
where $A$, $C$, and $\mathbf{H}$ are as in assumption (a). Then $L$  as in \eqref{defofLdivform} writes as (cf. Remark \ref{ex1})
\begin{eqnarray}\label{defofL}
Lf & = & \frac12 \sum_{i,j=1}^{d}a_{ij}\partial_i \partial_jf+\sum_{i=1}^{d}g_i\partial_i f, \quad f \in C_0^{\infty}(\R^d),
\end{eqnarray}
where
$$
g_i=\frac12\sum_{j=1}^{d} \partial_{j} (a_{ij}+c_{ji})+h_i,\quad  1 \leq i \leq d.
$$
\begin{theo}[Existence of an infinitesimally invariant measure] \label{eim}
Suppose  assumption \emph{(a)} holds. Then there exists $\rho \in H_{loc}^{1,p}(\R^d) \cap C(\R^d)$ with $\rho(x)>0$ for all $x \in \R^d$ such that
\begin{equation} \label{inv}
\int_{\R^d} L\varphi \, \rho dx=0, \quad \text{ for all } \varphi \in C_0^{\infty}(\R^d).
\end{equation}
\end{theo}
\begin{proof}
Using integration by parts, \eqref{inv} is equivalent to
\begin{equation} \label{weakdiv}
\int_{\R^d}  \langle \frac{1}{2} (A+C^{T}) \nabla \rho- \rho \,\mathbf{H} , \nabla \varphi  \rangle  dx = 0 \quad \text{ for all } \varphi \in C_0^{\infty}(\R^d).
\end{equation}
By \cite[Proposition 2.1.4, Corollary 2.1.6, Theorem 2.1.8]{BKRS} and Remark \ref{nonsym}, for every $n\in \N$, 
there exists a unique $v_n \in H_{0}^{1,p}(B_n) \cap C^{0,1-d/p}(\overline{B}_n)$ such that
\[
\int_{B_n} \langle  \frac{1}{2} (A+C^{T}) \nabla v_n -  v_n\, \mathbf{H}, \nabla \varphi  \rangle dx= \int_{B_n} \langle \mathbf{H}, \nabla \varphi  \rangle dx \; \; \text{ for all } \varphi \in C_0^{\infty}(B_n).
\]
Let $u_n:= v_n+1$. Then $u_n(x)=1$ for all $x \in \partial B_n$ and 
\[
\int_{B_n}  \langle  \frac{1}{2} (A+C^{T}) \nabla u_n - u_n\, \mathbf{H}, \nabla \varphi  \rangle dx=0, \quad \text{for all } \varphi \in C_0^{\infty}(B_n).
\]
Since $u_n^{-}\le v_n^-$, we see  $u_n^{-} \in H_0^{1,p}(B_n) \cap C^{0,1-d/p}(\overline{B}_n)$. Thus by 
Lemma \ref{subsol}, we get 
\[
\qquad \quad \int_{B_n} \langle  \frac{1}{2} (A+C^{T}) \nabla u_n^{-} -u_n^{-} \,\mathbf{H}, \nabla \varphi  \rangle dx \leq 0, \quad \text{for all }\varphi \in C_0^{\infty}(B_n), \,\varphi \geq 0.
\]
By \cite[Theorem 2.1.8]{BKRS} and Remark \ref{nonsym},  $u_n^{-} \leq 0$, so that $u_n \geq 0$. Suppose there exists $x_0 \in B_n$ with $u_n(x_0)=0$. Then, applying \cite[Corollary 5.2 (Harnack inequality)]{T73}  to $u_n$ on $B_n$, we get $u_n(x)=0$ for all $x\in  B_n$, which contradicts $u_n \in C^{0,1-d/p}(\overline{B}_n)$, since $u_n=1$ on $\partial B_n$. Hence $u_n(x)>0$ for all $x \in B_n$. Now let $\rho_n(x):= u_n(0)^{-1}u_n(x)$, $x\in B_n, n\in \N$. Then $\rho_n(0)=1$ and
\[
\int_{B_n}  \langle  \frac{1}{2} (A+C^{T}) \nabla \rho_n -\rho_n \mathbf{H}, \;\nabla \varphi  \rangle dx=0 \;\; \text{ for all } \varphi \in C_0^{\infty}(B_n).
\]
Fix $r>0$. Then, by \cite[Corollary 5.2]{T73}
\[
\sup_{x \in B_{2r}}\rho_n(x) \leq C_1 \inf_{x \in B_{2r}}\rho_n(x) \; \text{ for all } n>2r,
\]
where $C_1$ is independent of $\rho_n$, $n>2r$. Thus
\[
\sup_{x \in B_{2r}}\rho_n(x) \leq C_1 \; \text{ for all } n>2r.
\]
By \cite[Theorem 1.7.4]{BKRS} and Remark \ref{nonsym}
$$
\| \rho_n \|_{H^{1,p}(B_{r})} \leq C_2 \|\rho_n \|_{L^1(B_{2r})} \leq C_1 C_2 \,dx(B_{2r}), \; \text{ for all } n>2r,
$$
where $C_2$ is independent of $(\rho_n)_{n>2r}$.
By weak compactness of balls in $H_0^{1,p}(B_r)$ and the Arzela-Ascoli Theorem, there exist $(\rho_{n,r})_{n \geq 1} \subset (\rho_n)_{n >2r}$ and $\rho_{(r)} \in H^{1,p}(B_r) \cap C^{0,1-d/p}(\overline{B}_r)$
such that
$$
\rho_{n,r} \rightarrow \rho_{(r)} \;\text{ weakly in } H^{1,p}(B_r), \qquad \rho_{n,r} \rightarrow \rho_{(r)} \;\text{ uniformly on } \overline{B}_r.
$$
Considering $(\rho_{n,k})_{n \geq 1} \supset (\rho_{n,k+1})_{n \geq 1}$, $k\in \N$, we get $\rho_{(k)}=\rho_{(k+1)}$ on $B_{k}$, hence we can well-define $\rho$ as 
$$
\rho := \rho_{(k)} \text{ on } B_{k}, k\in \N. 
$$
Then $\rho \in H_{loc}^{1,p}(\R^d) \cap C(\R^d)$  with  $\rho(x) \geq 0$, $x\in \R^d$, $\rho(0)=1$ and for any $n \in \N$
\[
\int_{B_n} \langle  \frac{1}{2} (A+C^{T}) \nabla \rho -\rho \,\mathbf{H}, \nabla \varphi  \rangle dx=0 \;\; \text{ for all } \varphi \in C_0^{\infty}(B_n).
\]
By applying the Harnack inequality to $\rho$ on $B_r$ with $n>r$
\[
1=\rho(0) \leq \sup_{x \in B_{r}}\rho(x) \leq C_3 \inf_{x \in B_{r}}\rho(x),
\]
hence $\rho(x)>0$ for all $x \in B_r$. Therefore $\rho(x)>0$ for all $x \in \R^d$ and \eqref{inv} holds.
\end{proof}\\
From now on unless otherwise stated, we fix $\rho$ as in Theorem \ref{eim}. Set
$$
m:=\rho\,dx.
$$
Using integration by parts the following can be easily shown.
\begin{lem}\label{anti0}
If $Q:=(q_{ij})_{1 \leq i,j \leq d}$ is a $d \times d$ matrix of functions with $-q_{ji} =  q_{ij} \in H_{loc}^{1.2}(\R^d)\cap L^{\infty}_{loc}(\R^d)$, $1\le i,j\le d$. Then $\beta^{\rho, Q} \in L^2_{loc}(\R^d, \R^d,m)$ and $\beta^{\rho, Q}$ is weakly divergence free with respect to $m$, i.e. 
$$
\int_{\R^d} \langle \beta^{\rho, Q}, \nabla f  \rangle dm = 0, \quad \text{ for all } f \in C_0^{\infty}(\R^d).
$$
\end{lem}
\bigskip
Define
$$
\overline{\mathbf{B}}:= \mathbf{G}- \beta^{\rho, A+C^{T}}.
$$ 
Note that $\overline{\mathbf{B}}=\big (\mathbf{G}-\frac{1}{2}\nabla (A+C^{T})\big)-  \frac{(A+C^{T}) \nabla \rho}{2\rho} \in L^p_{loc}(\R^d, \R^d)$. Moreover,  
using \eqref{inv} and Lemma \ref{anti0}, we can see that $\beta^{\rho, C^{T}}+\overline{\mathbf{B}}\in L^2_{loc}(\R^d, \R^d,m)$ is weakly divergence free with respect to $m$, i.e. 
\begin{equation} \label{weakdiv2}
\int_{\R^d} \langle \beta^{\rho, C^{T}}+\overline{\mathbf{B}}, \nabla f  \rangle dm = 0 \quad \text{ for all } f \in C_0^{\infty}(\R^d).
\end{equation}
For $f,g \in C_0^{\infty}(\R^d)$, define
$$ 
\mathcal{E}^0(f,g): = \int_{\R^d} \langle A\nabla f, \nabla g \rangle \,dm.
$$
Then $(\mathcal{E}^0, C_0^{\infty}(\R^d))$ is closable in $L^2(\R^d,m)$. We denote its closure by $(\mathcal{E}^0, D(\mathcal{E}^0))$ and its associated generator by $(L^0, D(L^0))$. Since $C_0^{\infty}(\R^d) \subset D(L^0)_{0,b}$ we have that $D(L^0)_{0,b}$ is a dense subset of  $L^1(\R^d, m)$, and furthermore
$$
L^0 f  = \frac12\text{trace}(A\nabla^2 f)+ \langle \beta^{\rho, A}, \nabla f \rangle \in L^2(\R^d, m) \quad \text{ for all } f \in C_0^{\infty}(\R^d).
$$
Define
$$
Lf = L^0 f + \langle \beta^{\rho, C^{T}}+\overline{\mathbf{B}}, \nabla f \rangle, \;\;  f \in D(L^0)_{0,b}.
$$
Then $(L,D(L^0)_{0,b})$ is an extension of  $(L,C_0^{\infty}(\R^d))$ as defined in \eqref{defofL}. By \cite[Theorem 1.5]{St99}, there exists a closed extension $(\overline{L}, D(\overline{L}))$ of $(L,D(L^0)_{0,b})$
in $L^1(\R^d, m)$ which generates a sub-Markovian $C_0$-semigroup of contractions $(T_t)_{t>0}$ on $L^1(\R^d, m)$.  Restricting $(T_t)_{t> 0}$ to $L^1(\R^d,m)_b$, it is well-known that $(T_t)_{t> 0}$ can be extended to a sub-Markovian $C_0$-semigroup of contractions on each $L^r(\R^d,m)$, $r\in [1,\infty)$. Denote by $(L_r, D(L_r))$ the corresponding closed generator with graph norm
$$
\|f\|_{D(L_r)}:=\|f\|_{L^r(\R^d,m)}+ \|L_r f\|_{L^r(\R^d,m)},
$$
and by $(G_{\alpha})_{\alpha>0}$ the corresponding resolvent. For $(T_t)_{t>0}$ and $(G_{\alpha})_{\alpha>0}$ we do not explicitly denote in the notation on which $L^r(\R^d,m)$-space they act. We assume that this is clear from the context. Moreover, $(T_t)_{t>0}$ and $(G_{\alpha})_{\alpha>0}$ can be uniquely defined on $L^{\infty}(\R^d,m)$, but are no longer strongly continuous there.\\
For $f \in C_0^{\infty}(\R^d)$
\begin{eqnarray*}\label{41c}
\widehat{L}f &: =& L^0 f- \langle \beta^{\rho, C^{T}}+\overline{\mathbf{B}}, \nabla f \rangle= \frac12\text{trace}(A\nabla^2 f)+\langle\widehat{\mathbf{G}}, \nabla f \rangle,
\end{eqnarray*}
with
$$
\widehat{\mathbf{G}}:=(\widehat{g}_1, \ldots, \widehat{g}_d)=2\beta^{\rho,A}-\mathbf{G} = \beta^{\rho,A+C}-\overline{\mathbf{B}} \in L^{2}_{loc}(\R^d, \R^d, m).
$$ 
We see that $L$ and $\widehat{L}$ have the same structural properties, i.e. they are given as the sum of a symmetric second order elliptic differential operator and a divergence free first order perturbation with same integrability condition with respect to the measure $m$. Therefore all what will be derived below for $L$ will hold analogously for $\widehat{L}$. Denote the operators corresponding to $\widehat{L}$ (again defined through \cite[Theorem 1.5]{St99}) by $(\widehat{L}_r, D(\widehat{L}_r))$ for the co-generator on $L^r(\R^d,m)$, $r\in [1,\infty)$, $(\widehat{T}_t)_{t>0}$ for the co-semigroup, $(\widehat{G}_{\alpha})_{\alpha>0}$ for the co-resolvent. By  \cite[Section 3]{St99}, we obtain a corresponding bilinear form with domain 
$D(L_2) \times L^2(\R^d,m) \cup L^2(\R^d,m) \times D(\widehat{L}_2)$ by
\[ 
{\mathcal{E}}(f,g):= \left\{ \begin{array}{r@{\quad\quad}l}
  -\int_{\R^d} L_2 f \cdot g \,dm & \mbox{ for}\ f\in D(L_2), \ g\in L^2(\R^d,m),  \\ 
            -\int_{\R^d} f\cdot \widehat{L}_2 g \,dm  & \mbox{ for}\ f\in L^2(\R^d,m), \ g\in D(\widehat{L}_2). \end{array} \right .
\] 
$\mathcal{E}$ is called the {\it generalized Dirichlet form associated with} $(L_2,D(L_2))$. Using integration by parts, it is easy to see that for $f,g\in C_0^{\infty}(\mathbb{R}^d)$
\begin{eqnarray}\label{gdf}
{\cal E}(f,g)&=& \frac{1}{2}\int_{\mathbb{R}^d}\langle A\nabla f,\nabla g\rangle \, dm-\int_{\mathbb{R}^d}\langle \beta^{\rho, C^T}+\overline{\mathbf{B}},\nabla f\rangle g\, dm   \nonumber\\
&=&\frac{1}{2}\int_{\mathbb{R}^d}\langle (A+C)\nabla f,\nabla g\rangle \, dm-\int_{\mathbb{R}^d}\langle \overline{\mathbf{B}},\nabla f\rangle g\, dm,
\end{eqnarray}
and
\begin{align*}
\displaystyle L_2 f &=&  \frac12 \sum_{i,j=1}^{d} a_{ij}\partial_i \partial_jf+\sum_{i=1}^{d}g_i\partial_i f= \frac12\text{trace}(A\nabla^2 f)+ \langle \beta^{\rho, A+C^{T}}, \nabla f \rangle +\langle \overline{\mathbf{B}}, \nabla f  \rangle,  \\
\displaystyle \widehat{L}_2 f&=&  \frac12 \sum_{i,j=1}^{d} a_{ij}\partial_i \partial_jf+\sum_{i=1}^{d} \widehat{g}_i\partial_i f=  \frac12\text{trace}(A\nabla^2 f)+ \langle \beta^{\rho, A+C}, \nabla f \rangle  - \langle \overline{\mathbf{B}}, \nabla f \rangle.
\end{align*}
\bigskip

\subsection{Regularity results for resolvent and semigroup }\label{two}
\begin{theo}\label{4.1}
Assume \emph{(a)}. Then 
$$
\rho G_{\alpha} g \in H^{1,p}_{loc}(\R^d), \quad \forall g \in 
\cup_{r\in [q,\infty]} L^r(\R^d,m),
$$ 
and for any open balls $B$, $B'$ with $\overline{B} \subset B'$,
\begin{equation*}
\| \rho\, G_{\alpha}g \|_{H^{1,p}(B)} \le c_0 \left (\| g \|_{L^q(B',m)} + \| G_{\alpha}g \|_{L^1(B',m)}\right ),
\end{equation*}
where $c_0$ is independent of $g$.
\end{theo}
\begin{proof}
Let $g \in C_0^{\infty}(\R^d)$ and $\alpha >0$. Then for all $\varphi \in C_0^{\infty}(\R^d)$,
\begin{eqnarray}\label{resolventregularity1}
\int_{\R^d} (\alpha- \widehat{L}_2) \varphi \cdot \big(G_{\alpha} g\big) \,dm = \int_{\R^d} \widehat{G}_{\alpha}(\alpha-\widehat{L}_2) \varphi  \cdot g \,dm=\int_{\R^d} \varphi  g \,dm. \vspace{-0.4em}
\end{eqnarray}
Note that $G_{\alpha} g \in D(\overline{L})_b \subset D(\mathcal{E}^0)$ by \cite[Theorem 1.5]{St99}. Since $\rho$ is locally bounded below and $A$ satisfies \eqref{uniformell}, we have $D(\mathcal{E}^0) \subset H_{loc}^{1,2}(\R^d)$ and it follows $\rho  G_{\alpha} g \in H_{loc}^{1,2}(\R^d)$.  Define
\begin{eqnarray}\label{fhat}
\widehat{\mathbf{F}}:= \frac{1}{2} \nabla (A+C)- \mathbf{\widehat{G}}= -\frac{(A+C) \nabla \rho}{2 \rho}+\overline{\mathbf{B}} \in L^p_{loc}(\R^d, \R^d).
\end{eqnarray}
Given any open ball $B''$ and $\varphi \in C_0^{\infty}(B'')$, we have using integration by parts 
in the left hand side of \eqref{resolventregularity1}
\begin{eqnarray}\label{resolventregularity2}
\int_{B''}   \left [\langle \frac12 (A+C) \nabla (\rho G_{\alpha}g)  + ( \rho G_{\alpha}g )\widehat{\mathbf{F}}, \nabla \varphi \rangle + \alpha ( \rho G_{\alpha} g) \varphi \right ]dx  = \int_{B''}  (\rho g ) \varphi  dx.
\end{eqnarray}
By \cite[Theorem 1.8.3]{BKRS} and Remark \ref{nonsym}, for any open ball $B'$ with $\overline{B'} \subset B''$, we have $\rho G_{\alpha}g \in H^{1,p}(B')$.  Thus by \cite[Theorem 1.7.4]{BKRS} and Remark \ref{nonsym}, we obtain for any open ball $B$ with $\overline{B} \subset B'$, $r\in [q, \infty)$
\begin{eqnarray}\label{rhoresolventest}
\| \rho G_{\alpha} g \|_{H^{1,p}(B)} &\le& c_1 \left (\|\rho g \|_{L^q(B',dx)} + \|\rho G_{\alpha}g \|_{L^1(B',dx)}\right )\nonumber  \\
&\leq& \underbrace{c_1 (\sup_{B'} \rho^{\frac{q-1}{q}}  \vee 1)}_{=:c_0} \big( \|g\|_{L^q(B', m)} +  \| G_{\alpha}g \|_{L^1(B',m)  } \big) \\ \nonumber
\end{eqnarray}
By denseness of $C_0^{\infty}(\R^d)$ in $L^r(\R^d, m)$, \eqref{rhoresolventest} extends to $g\in L^r(\R^d,m)$, $r \in [q, \infty)$. 
For $g \in L^{\infty}(\R^d, m)$, let $g_n:=g 1_{B_n} \in L^{q}(\R^d, m), n\ge 1$. Then 
$\|g-g_n\|_{L^q(B', m)} + \| G_{\alpha}(g-g_n) \|_{L^1(B',m)} \to 0$ as $n\to \infty$. Hence \eqref{rhoresolventest} also extends to  $g\in L^\infty(\R^d,m)$.
\end{proof}

\begin{rem} \label{regular1}
\cite[Proposition 3.6]{LT18} holds in our more general situation with exactly the same proof.
\end{rem}
\begin{theo}\label{1-3reg} 
Assume \emph{(a)}. For each $s \in [1, \infty]$, consider the $L^s (\R^d, m)$-semigroup $(T_t)_{t>0}$. Then for any $f \in L^s(\R^d, m)$ and $t>0$, $T_t f$ has a locally H\"{o}lder continuous $m$-version $P_t f$ on $\R^d$. More precisely, $P_{\cdot}f(\cdot)$ is locally parabolic H\"{o}lder continuous on $\R^d \times (0, \infty)$ and for any bounded open sets $U$, $V$ in $\R^d$ with $\overline{U} \subset V$ and $0<\tau_3<\tau_1<\tau_2<\tau_4$, i.e. $[\tau_1, \tau_2] \subset (\tau_3, \tau_4)$, we have for some $\gamma\in (0,1)$ the following estimate for all $f \in \cup_{s\in[1,\infty]} L^s(\R^d, m)$ with $f \geq 0$,
\begin{equation} \label{thm main est} 
\|P_{\cdot} f(\cdot)\|_{C^{\gamma; \frac{\gamma}{2}}(\overline{U} \times [\tau_1, \tau_2])} \leq  C_6 \| P_{\cdot} f(\cdot) \|_{L^1( V \times (\tau_3, \tau_4), m\otimes dt) },
\end{equation}
where $C_6, \gamma$ are constants that depend on $\overline{U} \times [\tau_1, \tau_2], V \times (\tau_3, \tau_4)$, but are independent of $f$. 
\end{theo}
\begin{proof}
The proof is similar to the corresponding proof in \cite[Theorem 3.8]{LT18}, but there are some subtle differences. First assume $f \in D(L_2) \cap D(L_q) \cap \mathcal{B}_b(\R^d)$ with $f \geq 0$. Set $u(x,t):=\rho(x) P_t f(x)$. Then $P_t f\in D(L_q)$ and $\rho \in C(\R^d)$ implies $u \in C\big( \R^d\times[0,\infty) \big)$  by Proposition \ref{regular1}(iii).
Let $T>0$ be arbitrary. Then for any $\varphi\in C^{\infty}_0(\R^d\times (0,T))$
\begin{eqnarray}\label{**}
0=-\int _0^T\int_{\R^d}  \left ( \partial _t \varphi +\widehat{L}_2\varphi  \right ) u\,  dxdt. 
\end{eqnarray}
Since $u \in H^{1,2}(O\times (0,T))$ for any bounded and open set $O\subset \R^d$, using integration by parts in the right hand term of \eqref{**}, we get
\begin{equation}\label{eqfhat}
0=\int_0^T\int_{\R^d} \left ( \frac{1}{2}\langle(A+C) \nabla u, \nabla \varphi \rangle + u \langle \widehat{\mathbf{F}}, \nabla \varphi \rangle -u\partial_t\varphi \right ) dxdt,
\end{equation}
where $\widehat{\mathbf{F}}$ is as in \eqref{fhat}. Then as in \cite[Theorem 3.8]{LT18} 
\begin{eqnarray}\label{rest} 
\hspace{-2em} \|P_{\cdot} f(\cdot)\|_{C^{\gamma; \frac{\gamma}{2}}(\overline{U} \times [\tau_1, \tau_2])} & \leq& \| \rho^{-1} \|_{C^{0,\gamma}(\overline{U} )}   \|\rho(\cdot) P_{\cdot} f(\cdot)\|_{C^{\gamma; \frac{\gamma}{2}}(\overline{U} \times [\tau_1, \tau_2])}  \nonumber    \\
&\leq&  \underbrace{\| \rho^{-1} \|_{C^{0,\gamma}(\overline{U})}  C_2 C_5 }_{=:C_6} \| P_{\cdot} f(\cdot) \|_{L^1( V \times (\tau_3, \tau_4), m\otimes dt) }  \nonumber \\ 
& \leq&  C_6  (\tau_4-\tau_3) \|\rho \|_{L^1(V)}^{\frac{s-1}{s}}  \|f\|_{L^s(\R^d, m) }, \quad s \in [1, \infty],
\end{eqnarray}
where $\gamma$, $C_2$, $C_5$, are as in \cite[Theorem  3.8]{LT18}. \\
For $f \in L^1(\R^d, m) \cap L^{\infty}(\R^d, m)$ with $f \geq 0$ let $f_n:= n G_{n} f$. Then $f_n \in D(L_2) \cap D(L_q) \cap \mathcal{B}_b(\R^d)$ with $f_n \geq 0$ and $f_n \rightarrow f$ in $L^s(\R^d, m)$ for any $s \in [1, \infty)$.
Thus \eqref{rest} including all intermediate inequalities extend to $f \in L^1(\R^d, m) \cap L^{\infty}(\R^d, m)$ with $f \geq 0$. If $f \in L^s(\R^d, m)$, $f \geq 0$ and $s \in [1, \infty)$, let $f_n:= 1_{B_n} \cdot (f \wedge n)$.  Then $f_n \in  L^1(\R^d, m) \cap L^{\infty}(\R^d, m)$ with $f_n \geq 0$ and $f_n \rightarrow f$ in $L^s(\R^d, m)$. Thus \eqref{rest} including all intermediate inequalities extend to $f \in L^s(\R^d, m)$ with $f \geq 0$. For $f \in L^{\infty}(\R^d, m)$, the result follows exactly as in \cite[Theorem 3.8]{LT18}.
\end{proof}
\begin{rem}
Besides the possible non-symmetry of $A+C$ (that also occurs in $\widehat{\mathbf{F}}$), the difference between the proof of \cite[Theorem 3.8]{LT18} and Theorem \ref{1-3reg} is the approximation method. The proof of \cite[Theorem 3.8]{LT18} uses the denseness of  $C_0^{\infty}(\R^d)$ in $L^1(\R^d, m)$. The proof of Theorem \ref{1-3reg} uses the denseness of  $\cup_{\alpha>0} G_{\alpha}\big( L^1(\R^d,m)\cap L^{\infty}(\R^d, m)\big )$ in $L^1(\R^d, m)$.  Using the latter, we can get the corresponding result to \cite[Lemma 4.6]{LT18}  in the following Lemma \ref{irreduci0}.
\end{rem}
\begin{lem}\label{irreduci0}
Assume \emph{(a)}. Then:
\begin{itemize}
\item [(i)] Let $A \in \mathcal{B}(\R^d)$ be such that $P_{t_0} 1_A (x_0)=0$ for some $t_0>0$ and $x_0\in \R^d$. Then $m(A)=0$.
\item [(ii)] Let $A \in \mathcal{B}(\R^d)$ be such that $P_{t_0} 1_A (x_0)=1$ for some $t_0>0$ and $x_0\in \R^d$.  Then $P_t 1_{A}(x)=1$ \;for all $(x,t) \in \R^d \times (0,\infty)$.
\end{itemize}
\end{lem}
\begin{proof}
(i) Suppose $m(A)>0$. Choose an open ball $B_{r}(x_0)\subset \R^d$ such that
$$
0<m\left(A\cap B_{r}(x_0)\right)< \infty. 
$$
Let $u:=\rho P_{\cdot} 1_{A \cap B_{r}(x_0)}$. Then $0=u(x_0, t_0) \leq \rho(x_0) P_{t_0} 1_{A}(x_0)=0$. Set $f_n:=nG_{n}1_{A \cap B_{r}}$. Then $f_n \in  D(L_2) \cap D(L_q) \cap \mathcal{B}_b(\R^d)$ with $f_n \geq 0$ such that $f_n \rightarrow 1_{A \cap B_{r}(x_0)}$ in $L^1(\R^d, m)$. Let $u_n:=\rho P_{\cdot} f_n$. Fix $T>t_0$ and $U\supset \overline{B}_{r+1}(x_0)$. Since $u_n \in H^{1,2}(U \times (0,T))$ satisfies \eqref{**} (see proof of Theorem \ref{1-3reg}), \eqref{eqfhat} holds with $u$ replaced by $u_n$ for all $\varphi \in C_0^{\infty}(U\times (0, T))$. The rest of the proof is then exactly as in \cite[Lemma 4.6(i)]{LT18}.\\
(ii) Let $y \in \R^d$ and $0<s<t_0$ be arbitrary but fixed and let $r:=2\|x_0-y\|$ and let $B$ be any open ball. Take $g_n:=nG_n 1_{B \cap A}$. Then $g_n \in   D(L_2) \cap D(L_q) \cap \mathcal{B}_b(\R^d)$ with $0 \leq g_n \leq 1$ satisfying $g_n \rightarrow 1_{A \cap B}$ in $L^1(\R^d, m)$. The rest of the proof is now exactly as in \cite[Lemma 4.6 (ii)]{LT18}.
\end{proof}\\

\begin{rem}\label{irreduci}
Using the Lemma \ref{irreduci0}, \cite[Corollary 4,8]{LT18} holds in our more general situation with exactly the same proof.
\end{rem}
\begin{rem}\label{expg}
(i) (cf. Remark 4.5 in \cite{LT18}) Consider $A$, $C$, $\rho$, $\widetilde{\mathbf{B}}$ which are explicitly given by following assumptions. Let $A=(a_{ij})_{1 \leq i,j \leq d}$ be a matrix of functions as in assumption \emph{(a)} and $C=(c_{ij})_{1 \leq i,j \leq d}$ be a matrix of functions satisfying $c_{ij} =-c_{ij} \in H_{loc}^{1,2}(\R^d) \cap C(\R^d)$. Suppose that for some $p>d$, we are given $\rho \in H_{loc}^{1,p}(\R^d)\cap C(\R^d)$, $\rho(x)>0$ for all $x \in \R^d$, such that for some
$\widetilde{\mathbf{B}} \in L_{loc}^p(\R^d, \R^d)$ it holds
\begin{equation}\label{divfree20}
\int_{\R^d} \langle \widetilde {\mathbf{B}}, \nabla f  \rangle \rho dx = 0 \quad \text{ for all } f \in C_0^{\infty}(\R^d).
\end{equation}
Let
$$
\widetilde{L}f= L^0 f + \langle \beta^{\rho, C^{T}}+\widetilde{\mathbf{B}}, \nabla f \rangle, \;\;  f \in D(L^0)_{0,b}.
$$
Then \eqref{inv} holds for $L$ replaced with $\widetilde{L}$. Moreover, everything that was developed for $(L,D(L^0)_{0,b})$ right after Theorem \ref{eim} until and including Corollary \ref{irreduci} (and even beyond until the end of this article if additionally $\beta^{\rho, C^{T}}+\widetilde{\mathbf{B}} \in L_{loc}^q(\R^d, \R^d)$, i.e. assumption \emph{(b)} holds, cf. Remark \ref{mostresults}) holds analogously for $(\widetilde{L},D(L^0)_{0,b})$. 
Now suppose again that assumption \emph{(a)} holds. Then by Theorem \ref{eim}, there exists $\rho$ as right above such that $\widetilde{\mathbf{B}}:=\overline{\mathbf{B}}=\frac{1}{2}\nabla (A+C^{T})+ \mathbf{H}-\beta^{\rho, A+C^{T}}\in L_{loc}^p(\R^d, \R^d)$ and such that $\widetilde {\mathbf{B}}$ satisfies \eqref{divfree20}. Thus all that has been done up to now is in fact a special realization of the just explained explicit case. \\ \\
(ii) (cf. Remark 3.3  in \cite{LT18}) It is possible to realize the results of this article with $\R^d$ replaced by an arbitrary open set $U\subset \R^d$. Moreover as it is well-known the $L^p_{loc}$-condition can be relaxed by an $L^{p_n}_{loc}$-condition on an exhaustion $(V_n)_{n\in \N}$ of $\R^d$ (or $U$), where $p_n>d$ for all $n\in \N$ and $\lim_{n\to \infty} p_n=d$. 
\end{rem}

\section{Probabilistic results}\label{4}
\subsection{The underlying SDE}\label{three}
Additionally to assumption (a) we assume throughout this section that assumption (b) holds.  Then $C_0^2(\R^d) \subset D(L_1) \cap  D(L_q)$ and assumption ${(\textbf{H2})^{\prime}}$ of \cite{LT18} holds.
Here, assumption (b) was needed to get the continuity property of the resolvent in ${(\textbf{H2})^{\prime}}$(ii)  of \cite{LT18}. Thus, exactly as in \cite[Theorem 3.12]{LT18}, we arrive at  the following theorem:
\begin{theo}\label{existhunt}
There exists a Hunt process
\[
\M =  (\Omega, \F, (\F_t)_{t \ge 0}, (X_t)_{t \ge 0}, (\P_x)_{x \in \R^d\cup \{\Delta\}}   )
\]
with state space $\R^d$ and life time 
$$
\zeta=\inf\{t\ge 0\,:\,X_t=\Delta\}=\inf\{t\ge 0\,:\,X_t\notin \R^d\}, 
$$
having the transition function $(P_t)_{t \ge 0}$ as transition semigroup, such that $\M$ has continuous sample paths in the one point compactification $\R^d_{\Delta}$ of $\R^d$ with the cemetery $\Delta$ as point at infinity.
\end{theo}
\begin{rem}\label{mostresults}
Actually, under assumptions \emph{(a)} and \emph{(b)} most of the results from \cite{LT18} generalize to the more general coefficients considered here, i.e. the analogues of Lemmas 3.14, 3.15, 3.18, Propositions 3.16, 3.17, Theorem 3.19, Remark 3.20 and the analogues of the results in Chapter 4 of \cite{LT18} hold. These results include, various non-explosion criteria, moment inequalities, a general Krylov type estimate, recurrence criteria and moreover (by combining our results with results of \cite{DPZB} and \cite{BKRS}, see \cite[Theorem 4.15, Proposition 4.17]{LT18}) criteria for ergodicity including uniqueness of the invariant probability measure $\rho dx$. 
\end{rem}
According to Remark \ref{mostresults}, we obtain:
\begin{theo}\label{weakexistence}
Consider the Hunt process $\M$ from Theorem \ref{existhunt} with coordinates $X_t=(X_t^1,\ldots,X_t^d)$. Let $(\sigma_{ij})_{1 \le i \le d,1\le j \le l}$, $l\in \N$ arbitrary but fixed, be any matrix consisting of continuous functions  $\sigma_{ij}\in C(\R^d)$ for all $i,j$,  such that $A=\sigma\sigma^T$,  i.e. 
$$
a_{ij}(x)=\sum_{k=1}^l \sigma_{ik}(x)\sigma_{jk}(x), \ \  \forall x\in \R^d, \ 1\le i,j\le d.
$$
Then on a standard extension 
of $(\Omega, \F, (\F_t)_{t\ge 0}, \P_x )$, $x\in \R^d$, that we denote for notational convenience again 
by $(\Omega, \F, (\F_t)_{t\ge 0}, \P_x )$, $x\in \R^d$, there exists a standard  $l$-dimensional Brownian motion $W = (W^1,\ldots,W^l)$ starting from zero such that 
$\P_x$-a.s. for any $x=(x_1,\ldots,x_d)\in \R^d$, $i=1,\ldots,d$
\begin{equation}\label{weaksolution}
X_t^i = x_i+ \sum_{j=1}^l \int_0^t \sigma_{ij} (X_s) \, dW_s^j +   \int^{t}_{0}   g_i(X_s) \, ds, \quad 0\le  t <\zeta.
\end{equation}
\end{theo}
The non-explosion result and moment inequality of order $1$ in the following theorem is new and allows for linear growth together with $L^q(\R^d,m)$-growth and singularities of the drift. However, the growth condition on the dispersion coefficient is unusually of square root order, but can allow $L^p(\R^d,m)$-growth.
The theorem complements various other non-explosion results from \cite{LT18} and existing literature. And it complements in particular \cite[Theorem 4.4]{LT18}, where a usual linear growth condition (that however does not allow for $L^q(\R^d,m)$-singularities of the drift) on dispersion and drift coefficients is used to show moment inequalities of orders $s>0$ and $s=2$. 
\begin{theo}\label{supestimate}
Let $\sigma=(\sigma_{ij})_{1\le i,j\le d}$ be as in Theorem \ref{weakexistence}, i.e. $l=d$ (such $\sigma$ always exists, cf. \cite[Lemma 3.18]{LT18}) and assume that for some $h_1 \in L^p(\R^d,m)$, $h_2 \in L^q(\R^d,m)$ and $C>0$ it holds for a.e. $x \in \R^d$
\begin{eqnarray*} \label{lineargrowth}
\max_{1\leq i,j \leq d}|\sigma_{ij}(x)| \leq |h_1(x)|+ C(\sqrt{\|x\|}+1), \quad \max_{1 \leq i \leq d}|g_i(x)| \leq |h_2(x)| + C(\|x\|+1).
\end{eqnarray*}
Then $\M$ is non-explosive and for any $T>0$, and any open ball $B$, there exist constants $C_{5,T}$, depending in particular on $B$, and $C_6$ such that
\[
\sup_{x \in \overline{B}}\E_{x}\left[\sup_{s \leq t} \|X_s \|\right] \leq  C_{5,T}\cdot e^{C_6 \cdot t}, \quad \forall t\le  T.
\]
\end{theo}
\begin{proof}
Let $x \in \overline{B}$ and $n \in \N$ such that $x \in B_n$ ($B_n$ is the open ball about zero with radius $n$ in $\R^d$). Let $0 \leq t \leq T$. 
Then with  $\sigma_n:=\inf\{t>0\,:\, X_t\in \R^d\setminus B_n\}$, $n\ge 1$, we obtain $\P_{x}$-a.s. for any $1\le i\le d$
\begin{eqnarray*} 
 \sup_{0 \leq s \leq t \wedge \sigma_{n}} |X_s^i|
\leq  |x_i| + \sum_{j=1}^d \sup_{\;0 \leq s \leq t \wedge \sigma_{n} } \left |\int_0^{s} \sigma_{ij} (X_u) \, dW_u^j \right | + \sup_{\;0 \leq s \leq t \wedge \sigma_{n} }  \int^{ s}_{0}   |g_i(X_u)| \, du.
\end{eqnarray*}
By the Burkholder-Davis-Gundy inequality \cite[Chapter IV. (4.2) Corollary]{RYor99} and \eqref{resest}, there exists a constant $C_{3,T}$, depending on $\|h_1 \|_{L^{2q}(\R^d, m)}$ and $B$, and constants $c$, $C$, such that
\begin{eqnarray} \label{sdeest2}
&& \sum_{j=1}^{d} \E_x \left [ \sup_{\;0 \leq s \leq t \wedge \sigma_{n} } \left |\int_0^{s} \sigma_{ij} (X_u) \, dW_u^j \right |   \right ] \leq \sum_{j=1}^{d} c    \E_x \left [  \int_0^{t \wedge  \sigma_{n} } \sigma_{ij}^2(X_{u })  du   \right ]^{1/2}  \nonumber \\
&\leq&  C_{3,T}+ C\sqrt{3}cd  \int_0^ t \E_{x} \left[ \sup_{0 \leq s \leq u \wedge \sigma_n} \|X_s\| \right] du, \nonumber
\end{eqnarray}
and 
\begin{eqnarray*}
&&\E_{x}\left[\sup_{0 \leq s \leq t \wedge \sigma_{n}}\int_0^{s} |g_i(X_u)|du\right] \\
 &\leq&  \underbrace{e^T c_{B,q,T} \|h_2 \|_{L^q(\R^d, m)}+CT}_{=:C_{4,T}}+  C\int_0^{t }\E_{x} \left[  \sup_{0 \leq s \leq u \wedge \sigma_n} \|X_s\|\right] du.
\end{eqnarray*}
Hence, for some constants $C_{5,T}$ and $C_6$
\begin{eqnarray*}
\E_x\left[\sup_{0 \leq s \leq t \wedge \sigma_n}\|X_s\| \right]  
&\leq& C_{5,T}+ C_6 \int_0^{t }\E_{x} \left[  \sup_{0 \leq s \leq u \wedge \sigma_n} \|X_s\|\right] du. \label{gwi}
\end{eqnarray*}
Now let $p_n(t):=\E_{x} \left[  \sup_{0 \leq s \leq t\wedge \sigma_{n}}\|X_{s}\| \right]$. Then by \eqref{gwi}, we obtain
\begin{eqnarray*}
p_n(t) & \leq &C_{5,T}+ C_6\int_{0}^{t} p_n(u) du, \;\; \quad 0\leq t \leq T.
\end{eqnarray*}
By Gronwall's inequality, $p_n(t) \leq C_{5,T}\cdot e^{C_6 \cdot t}$ for any $t \in [0,T]$. Using in particular the Markov inequality,
\begin{eqnarray*}
\P_x(\sigma_n \leq T)
\leq \frac{1}{n} C_{5,T}\cdot e^{C_6 \cdot T}.
\end{eqnarray*}
Therefore, letting $n\to \infty$ and using the analogue of Lemma 3.15(i) in \cite{LT18} (cf. Remark \ref{mostresults}), we obtain $\P_{x}(\zeta= \infty)=1$. Finally applying Fatou's lemma to $p_n(t)$, we obtain
\[
\E_{x}\left[\sup_{s \leq t} \|X_s \|\right] \leq  C_{5,T}\cdot e^{C_6 \cdot t}, \quad \forall t\le  T.
\]
Since the right hand side does not depend on $x\in \overline{B}$ the assertion follows.
\end{proof}
\begin{exam}\label{infsin}
Let $\eta \in C_0^{\infty}(B_{1/4})$ be given.
Define $w:\R^d \rightarrow \R$ by
$$
w(x_1, \dots, x_d):= \eta(x_1, \dots, x_d) \cdot  \int_{-2}^{x_1} \frac{1}{|y_1|^{1/d}} 1 _{[-1,1]}(y_1) dy_1.
$$
Then $w \in H^{1,q}(\R^d) \cap C_0(B_{1/4})$ but $\partial_1 w \notin L_{loc}^{d}(\R^d)$. Define $v:\R^d \rightarrow \R$ by
$$
v(x_1, \dots, x_d):= w(x_1, \dots, x_d)+ \sum_{i=1}^{\infty} \frac{1}{2^i} w(x_1-i, \dots, x_d)
$$
Then $v \in H^{1,q}(\R^d) \cap C(\R^d)$ but $\partial_1 v \notin L_{loc}^{d}(\R^d)$. Now define $P=(p_{ij})_{1 \leq i, j \leq d}$ as
$$
p_{1d}:=v, \;\; p_{d1}:=-v,\;\quad  p_{ij}:=0 \text{ if }\;(i,j) \notin \{(1,d), (d,1)\}.
$$
Let $Q=(q_{ij})_{1 \leq i, j \leq d}$ be a matrix of functions such that $q_{ij}=-q_{ij} \in H^{1,q}_{loc}(\R^d) \cap C(\R^d)$ for all $1\leq i,j \leq d$ and assume there exists a constant $C>0$ satisfying
$$
\| \nabla Q \| \leq C( \|x\|+1), \quad \text{ for a.e. on } \R^d. 
$$
Let $A:=id$, $C:=P+Q$ and $\mathbf{H}\equiv0$. Then $A$ and $C$ satisfy assumption \emph{(a)} with $\mathbf{G}:= \frac12 \nabla (A+C^{T})$ and assumption \emph{(b)} is satisfied. Define $\rho\equiv 1$ on $\R^d$. Then $\rho$ satisfies \eqref{inv} and $\overline{\mathbf{B}}\equiv0$. Obviously $\sigma=id$ and $\mathbf{G}$ satisfy the conditions of Theorem \ref{supestimate}. Thus $\M$ from Theorem  \ref{existhunt} is non-explosive. Note that the non-explosion criterion of this example can not be  derived from  \cite[Proposition 1.10]{St99}, nor from \cite[(3)]{LT18} or for instance \cite[Assumption 2.1]{GyMa} (one of the pioneering works on local and global well-posedness of SDEs with unbounded merely measurable drifts), since $\mathbf{G}$ has a part with infinitely many singular points outside an arbitrarily large compact set and may have a part with linear growth. 
\end{exam}

\subsection{Uniqueness in law under low regularity}\label{4.2}
Let $\widetilde \M =  (\widetilde\Omega, \widetilde \F, (\widetilde X_t)_{t \ge 0}, (\widetilde\P_x)_{x \in \R^d\cup \{\Delta\}})$ be a right process (see for instance \cite{Tr5}).  For a $\sigma$-finite or finite Borel measure $\nu$ on $\R^d$ we define
$$
\widetilde \P_\nu(\cdot):=\int_{\R^d} \widetilde \P_x(\cdot)\,\nu(dx).
$$ 
Consider $(L,C_0^{\infty}(\R^d))$ as defined in  \eqref{defofL}. According to \cite[Definition 2.5]{St99}, we define:
\begin{defn}
A right process $\widetilde \M =  (\widetilde\Omega, \widetilde \F, (\widetilde X_t)_{t \ge 0}, (\widetilde\P_x)_{x \in \R^d\cup \{\Delta\}})$  with state space $\R^d$ and natural filtration $(\widetilde{\F}_t)_{t \ge 0}$ is said to solve the martingale problem for $(L,C_0^{\infty}(\R^d))$, if for all $u\in C_0^{\infty}(\R^d)$:
\begin{itemize}
\item[(i)] $\int_0^t L u(\widetilde X_s) \, ds$, $t \ge 0$, is $\widetilde\P_m$-a.e. independent of the measurable $m$-version chosen for $Lu$.
\item[(ii)] $u(\widetilde X_t) - u(\widetilde X_0) - \int_0^t L u(\widetilde X_s) \, ds$, $t \ge 0$,
is  a continuous $(\widetilde{\F}_t)_{t \ge 0}$-martingale under $\widetilde\P_{vm}$  for any $v \in \mathcal{B}_b^+(\R^d)$ such that $\int_{\R^d}v\,dm=1$.
\end{itemize}
\end{defn}
\begin{defn}\label{subinvariant}
A $\sigma$-finite Borel measure $\nu$ on $\R^d$ is called sub-invariant measure for a right process $\widetilde \M =  (\widetilde\Omega, \widetilde \F, (\widetilde X_t)_{t \ge 0}, (\widetilde\P_x)_{x \in \R^d\cup \{\Delta\}})$  with state space $\R^d$, if
\begin{eqnarray}\label{le}
\int_{\R^d} \widetilde\E_x[f(\widetilde X_t)] \nu(dx)\le \int_{\R^d} f(x) \nu(dx)
\end{eqnarray}
for any $f\in L^1(\R^d,\nu)\cap \mathcal{B}_b(\R^d)$, $f\ge 0$, $t\ge 0$. $\nu$ is called invariant measure for $\widetilde \M$, if \lq\lq $\le$\rq\rq \ can be replaced by \lq\lq $=$\rq\rq\ in \eqref{le}
\end{defn}
Part (i) of the following proposition is proven in \cite[Proposition 2.6]{St99}. And part (ii) is a simple consequence of part (i), the strong Feller property of $(p^{\M}_t)_{t\ge 0}$, $\M$ as in Theorem \ref{existhunt}, and the fact that the law of a right process is uniquely determined by its transition function (and the initial condition).
\begin{prop}\label{wilhelm}
\begin{itemize}
\item[(i)] Let $\widetilde \M =  (\widetilde\Omega, \widetilde \F, (\widetilde X_t)_{t \ge 0}, (\widetilde\P_x)_{x \in \R^d\cup \{\Delta\}})$ 
solve the martingale problem for $(L,C_0^{\infty}(\R^d))$ such that $m$ is a sub-invariant measure for $\widetilde\M$ and let  $(L,C_0^{\infty}(\R^d))$ be $L^1$-unique. Then $p^{\widetilde\M}_t f(x):=\widetilde\E_x[f(\widetilde X_t)]$ is an $m$-version of $T_tf$ for all $f\in L^1(\R^d,m)\cap \mathcal{B}_b(\R^d)$, $t\ge 0$ and $m$ is an invariant measure for $\widetilde \M$. 
\item[(ii)] If additionally to the assumptions in (i), $(p^{\widetilde\M}_t)_{t\ge 0}$ is strong Feller, then $\widetilde \P_x=\P_x$ for any $x\in \R^d$.
\end{itemize}
\end{prop}
\begin{prop}\label{uniquenesslawgen}
Suppose that assumptions \emph{(a)} and \emph{(b)} hold, and that for any compact set $K$ in $\R^d$, there exist $L_K\ge 0, \alpha_K\in (0,1)$ with
\begin{eqnarray*}
|a_{ij}(x)   -a_{ij}(y)|\le L_K\|x-y\|^{\alpha_K} ,\quad        \forall  x,y\in K,\,1\le i,j\le d.
\end{eqnarray*}
Suppose further that $m$ is an invariant measure for $\M$. Let $\widetilde \M$ be a right process with strong Feller transition function $(p^{\widetilde\M}_t)_{t\ge 0}$ that  solves the martingale problem for $(L,C_0^{\infty}(\R^d))$ and such that $m$ is a sub-invariant measure for $\widetilde\M$. Then $\widetilde \P_x=\P_x$ for any $x\in \R^d$.
\end{prop}
\begin{proof} By \cite[Corollary 2.2]{St99} $(L,C_0^{\infty}(\R^d))$ is $L^1$-unique, if and only if $m$ is an invariant measure for $\M$. Then apply Proposition \ref{wilhelm}.
\end{proof}
\begin{rem}
Note that $m$ is an invariant measure for $\M$ as in Theorem \ref{existhunt}, if and only if the co-semigroup $(\widehat{T}_t)_{t>0}$  of  $(T_t)_{t>0}$ is conservative. One advantage of our approach is that we can use all previously derived conservativeness results for generalized Dirichlet forms (see for instance \cite[Proposition 1.10]{St99}, \cite{GT17}, \cite{LT18}, but also Example \ref{finalex}).
\end{rem}
\begin{exam}\label{finalex}
\begin{itemize}
\item[(i)]
Assume that assumptions \emph{(a)}  and \emph{(b)} hold and that the $a_{ij}$ are locally H\"{o}lder continuous on $\R^d$ as in Proposition \ref{uniquenesslawgen}. If there exists a constant  $C> 0$ and some $N_0\in \N$, such that 
\begin{eqnarray}\label{conservative9}
-\frac{\langle A(x)x, x \rangle}{ \left \| x \right \|^2 +1}+ \frac12\mathrm{trace}A(x)+ \big \langle \mathbf{G}(x), x \big \rangle \leq -C\left ( \left \| x \right \|^2+1\right )
\end{eqnarray}
for a.e. $x\in \R^d\setminus B_{N_0}$, then $\M$ as in Theorem \ref{existhunt} solves the martingale problem for $(L,C_0^{\infty}(\R^d))$ and $m$ is an invariant measure for $\M$ by the analogue of \cite[Proposition 4.17]{LT18} (see Remark \ref{mostresults}). In this situation Proposition \ref{uniquenesslawgen} applies.
\item[(ii)]
Let $A$, $C$ and $\mathbf{G}$ be as in Example \ref{infsin}. By Theorem \ref{supestimate}, not only $\M$ but also its co-process $\widehat{\M}$ is non-explosive. Hence $dx$ is an invariant measure for $\M$. Now if $a_{ij}$ are locally H\"{o}lder continuous on $\R^d$ as in Proposition \ref{uniquenesslawgen} then Proposition \ref{uniquenesslawgen} also applies.
\item[(iii)] Suppose that in the situation of Remark \ref{expg}(i) the conditions of \cite[Theorem 4.11]{LT18} hold with $\overline{\mathbf{B}}=\widetilde{\mathbf{B}}$ and that the $a_{ij}$ are locally H\"{o}lder continuous on $\R^d$ as in Proposition \ref{uniquenesslawgen}. Then $\rho
\, dx$ is an invariant measure for $\M$ and Proposition \ref{uniquenesslawgen} again applies.
\end{itemize}
\end{exam}

\centerline{}
Haesung Lee, Gerald Trutnau\\
Department of Mathematical Sciences and \\
Research Institute of Mathematics of Seoul National University,\\
1 Gwanak-Ro, Gwanak-Gu,
Seoul 08826, South Korea,  \\
E-mail: fthslt14@gmail.com, trutnau@snu.ac.kr
\end{document}